\documentclass[10pt]{article}
\usepackage{amsfonts}
\usepackage{graphicx}
\usepackage{amsmath}
\usepackage{amsthm}
\usepackage{hyperref}
\usepackage{amssymb}
\usepackage{amscd}
\usepackage{color}

\usepackage{graphicx}
\usepackage{microtype}
\usepackage{enumerate}
\usepackage{fullpage}
\usepackage{pinlabel}
\usepackage[top=1.1in,bottom=1.3in,left=1.3in,right=1.3in]{geometry}
\usepackage{titling}

\theoremstyle{definition}
\newtheorem{theorem}{Theorem}[section]
\newtheorem*{theorem*}{Theorem}

\newtheorem{construction}[theorem]{Construction}
\newtheorem{corollary}[theorem]{Corollary}
\newtheorem{example}[theorem]{Example}
\newtheorem{lemma}[theorem]{Lemma}

\newtheorem{proposition}[theorem]{Proposition}
\newtheorem{remark}[theorem]{Remark}
\newtheorem{question}[theorem]{Question}
\newtheorem{definition}[theorem]{Definition}
\newtheorem*{ackn}{Acknowledgements}

\theoremstyle{plain}

\newcommand{\R}{\mathbb{R}}
\newcommand{\Z}{\mathbb{Z}}
\newcommand{\N}{\mathbb{N}}

\DeclareMathOperator{\id}{id}
\DeclareMathOperator{\PL}{PL}
\DeclareMathOperator{\supp}{supp}

\DeclareMathOperator{\Homeo}{Homeo}
\DeclareMathOperator{\Diff}{Diff}
\DeclareMathOperator{\interior}{Int}

\title{Strong distortion in transformation groups}
\author{Fr\'ed\'eric Le Roux and Kathryn Mann}
\date{}
\begin{document}

\maketitle

\abstract{We show that the groups $\Diff^r_0(\R^n)$ and $\Diff^r(\R^n)$ have the \emph{strong distortion property}, whenever $0 \leq r \leq \infty, r \neq n+1$. This implies in particular that every element in these groups is distorted, a property with dynamical implications. The result also gives new examples of groups with Bergman's strong boundedness property as in \cite{Bergman}.  
With related techniques we show that, for $M$ a closed manifold or homeomorphic to the interior of a compact manifold with boundary, the diffeomorphism groups $\Diff_0^r(M)$ satisfy a relative Higman embedding type property, introduced by Schreier.  In the simplest case, this answers a problem asked by Schreier in the famous \emph{Scottish Book}.   \\
\\
MSC classes: 22F05, 37C05, 57S25, 20F05

}

\section{Introduction}
\sloppy

It is a classical theorem of Higman, Neumann, and Neumann \cite{HNN} that every countable group can be realized as a subgroup of a group generated by two elements.   In this paper, we are concerned with a relative version of this property, inspired by the following question of Schreier.  

\begin{question}[Schreier (1935), Problem 111 in the \emph{Scottish Book} \cite{Scottish}] 
Does there exist an uncountable group with the property that every countable sequence of elements of this group is contained in a subgroup which has a finite number of generators? In particular, does the group $S_\infty$ of permutations of an infinite set, and the group of all homeomorphisms of the interval have this property?
\end{question} 

The first part of this question was answered positively, and using the example of $S_\infty$, by Galvin \cite{Galvin}, although the existence of such a group also follows easily from the earlier work of Sabbagh in \cite{Sabbagh}.  A few other examples of groups with this property have been found, see eg. \cite{Cor} and references therein.  
However, as of the 2nd (2015) edition of the Scottish book, the question concerning the group of homeomorphisms of the interval remains open.  
Here we give a positive answer to Schreier's question for the group of homeomorphisms of the interval, and show that the property in question holds for many other transformation groups as well.  For concreteness, say that a group $G$ has the \emph{Schreier property} if every countable subset of $G$ is contained in a finitely generated subgroup of $G$.  We prove: 

\begin{theorem} \label{schrier-thm} 
Let $0 \leq r \leq \infty$, and let $M$ be a $C^r$ manifold with $\dim(M) \neq r-1$, either closed or homeomorphic to the interior of a compact manifold with boundary.   Then the group $\Diff^r_0(M)$ of isotopically trivial diffeomorphisms of $M$ has the Schreier property.  

Consequently, the group $\Diff^r(M)$ has the Schreier property if and only if the mapping class group $\Diff^r(M)/\Diff^r_0(M)$ is finitely generated.  
\end{theorem}

The answer to Schreier's question is the special case $\Diff^0(\R) = \Homeo(\R) \cong \Homeo(I)$.
The assumption $\dim(M) \neq r-1$ in this theorem comes from the fact that the group $\Diff_c^r(M)$ of compactly supported diffeomorphisms of a manifold $M$ is known to be simple in this case, but the algebraic structure of $\Diff_c^r(M)$ is not understood when $\dim(M) = r-1$.  In particular, it is not known whether such a group admits a surjective homomorphism to $\R$, if so, it would fail to have the Schreir property.  

In many cases, it turns out that Schreier's property follows from a stronger dynamical property called \emph{strong distortion}.

\begin{definition}
A group is \emph{strongly distorted}\footnote{We follow the terminology of Cornulier. This property was called ``Property P" in \cite{Khelif}.}
 if there exists an integer $m$ and an integer-valued sequence $w_n$ such that, for every sequence $g_n$ in $G$, there exists a finite set $S$ of cardinality $m$, such that each element $g_n$ can be expressed as a word of length $w_n$ in $S$.  
\end{definition}

In particular, strong distortion implies that every element of $G$ is \emph{arbitrarily distorted} in the usual sense of distortion of group elements or subgroups.  This fact has important dynamical implications when $G$ is a group of homeomorphisms or diffeomorphisms of a manifold or more general metric space, as distortion places constraints on the dynamics of such transformations.  For example, the case of distorted diffeomorphisms of surfaces is studied in \cite{FH}.  

Closely related to strong distortion are the notions of \emph{strong boundedness}, also called \emph{property (OB)} or the \emph{Bergman property}, and \emph{uncountable cofinality}.  

\begin{definition}
A group $G$ is \emph{strongly bounded} if every function $\ell: G \to \R_{\geq 0}$, satisfying $\ell(g^{-1}) = \ell(g)$, $\ell(\id) = 0$, and the triangle inequality $\ell(gh) \leq \ell(g) + \ell(h)$, is bounded.
\end{definition}

\begin{definition}
A group $G$ has \emph{uncountable cofinality} if it cannot be written as the union of a countable strictly increasing sequence of subgroups. 
\end{definition}

It is not hard to see that the Schreier property implies uncountable cofinality, that strong distortion implies both strong boundedness and the Schreier property (we give quick proofs at the end of this introduction), and that strong boundedness is equivalent to the dynamical condition that every isometric action of $G$ on a metric space has bounded orbits (see the appendix to \cite{CF}).   Our second main result is the following.  

\begin{theorem} \label{distorted-thm}
The groups $\Diff^r_0(\R^n)$ and $\Diff^r(\R^n)$ are strongly distorted, for all $n$ and all $r \neq n+1$.   
\end{theorem} 

This is particularly surprising since $\Diff_c^r(\R^n)$, as well as the groups $\Diff^r_0(M)$ for compact $M$, are never strongly distorted, nor even strongly bounded, whenever $r \geq 1$.  This is also true of $\Diff^0_0(M) = \Homeo_0(M)$ provided that $M$ has infinite fundamental group -- this follows from \cite[Example 6.8]{CF}, or more explicitly from \cite[Prop. 20]{MR} which implies that all \emph{maximal metrics} on $\Homeo_0(M)$ are unbounded length functions.    
In particular, for these examples, there is no hope to improve Theorem \ref{schrier-thm} to a proof of strong boundedness or distortion.    

Interestingly the question of strong boundedness and strong distortion for homeomorphism groups of manifolds with \emph{finite} fundamental group, other than the spheres, remains open.   

The following table summarizes the results mentioned above.  

\bigskip

\begin{center}
    \begin{tabular}{|p{4.1cm} | p{1.5cm} | p{1.5cm} | p{1.5cm}| p{2cm}|}
    \hline
     & Strongly distorted & Strongly bounded & Schreier property & Uncountable cofinality \\ \hline
    $\Diff^r_0(\R^n)$, $r \neq n+1$ & $\checkmark$ & $\checkmark$ & $\checkmark$ & $\checkmark$
     \\ \hline
    $\Homeo_0(S^n)$  & $\checkmark$ & $\checkmark$ & $\checkmark$ & $\checkmark$
    \\
    (Cornulier \cite[Appendix]{CF}) & & & &
     \\ \hline
    $\Homeo_0(M)$, $|\pi_1(M)| < \infty$  & ?  & ? & $\checkmark$ & $\checkmark$
     \\ \hline
      $\Homeo_0(M)$, $|\pi_1(M)| = \infty$  & X\, \cite{CF}, \cite{MR}& X & $\checkmark$ & $\checkmark$
     \\ \hline
    $\Diff^r_0(M)$ $r\geq 1$, & X & X & $\checkmark$* & $\checkmark*$
    \\ 
     $M$ compact & & & &
     \\ \hline
    \end{tabular}
\end{center} 
 *under the hypotheses of Theorem \ref{schrier-thm}

\bigskip

Despite the results mentioned above, one should not expect that most transformation groups have Schreier's property.  For instance, we have the following.

\begin{example}[Failure of Schreier's property]
The group $\PL(M)$ of piecewise-linear homeomorphisms of a $\PL$ manifold $M$ does not have the Schreier property.  To see this directly, fix a system of $\PL$ charts for $M$, and note that for any finite symmetric set $S \subset G$, the set of all jacobians (at all points where defined) of elements of $S$ is a finite subset, say $F \subset \mathrm{GL}(n, \R)$.  Thus, for any element $g$ generated by $S$ and any point $x \in M$, the jacobian of $g$ at $x$ has each entry an algebraic expression in the (finite) set of entries of elements of $F$.  Thus, if $g_n$ agrees with dilation by $\lambda_n$ near some fixed point $x$, where $\lambda_n$ is a sequence of algebraically independent real numbers, then the sequence $\{g_n\}$ cannot be generated by any finite set.

As an easier example, suppose $G$ is the group of compactly supported homeomorphisms or diffeomorphisms of a noncompact manifold $M$.  Let $K_n$ be an exhaustion of $M$ by compact sets, with $K_n$ contained in the interior of $K_{n+1}$.  Then $G$ is the countable increasing union of the subgroups $G_n := \{g : g(x)=x \text{ for all } x \notin K_n\}$.  Thus, $G$ has countable cofinality, and hence does not have the Schreier property.
\end{example}

\begin{example}[Open question]
We do not know whether either of the groups $\Homeo_0(\mathbb{S}^2, \mathrm{area})$ or $\Diff^r_0(\mathbb{S}^2, \mathrm{area}), r \geq 1$ of area preserving homeomorphisms or diffeomorphism of the sphere have the Schreier property.  We do know that they are not strongly bounded.  In the case of diffeomorphisms, this follows from the fact that norm of the derivative gives an unbounded length function.  However, there is also another (conjugation-invariant) norm, the \emph{Viterbo norm} on $\Diff_0(\mathbb{S}^2, \mathrm{area})$, and by work of \cite{Seyfaddini} it extends to a norm on $\Homeo_0(\mathbb{S}^2, \mathrm{area})$.

On the other hand, the groups $\Diff^r_0(\mathbb{T}^2, \mathrm{area})$ do \emph{not} have Schreier's property. Indeed, the rotation number of the area is a morphism from these groups to $\mathbb{R}$; if a group $G$ has Schreier's property then it is also the case of its image under a morphism; and $\mathbb{R}$ does not have Schreier's property. The question is again open if we restrict to the kernels of these morphisms (that is, to the groups of \emph{Hamiltonian} diffeomorphisms or homeomorphisms).
\end{example}

\begin{remark}[A stronger property] 
As pointed out by G. Bergman, the proof of Theorem \ref{distorted-thm} shows that the group $G = \Diff^r_0(\R^n)$ has a stronger property; namely the following:  {\em There is an integer $m$ and a sequence $W_N$ of words in $m$ letters (elements of the free group on m generators) with the property that, for any sequence $\{f_n\}$ in $G$, there exists $s_1, ... , s_m \in G$ such that $f_n = W_n(s_1, s_2, ... s_m)$.  }

Bergman asks if this property is equivalent to strong boundedness.  We do not know a counterexample.  
\end{remark} 

\paragraph{Implications between properties.} 
We conclude these introductory remarks with some implications between properties that are not evident from the table given above.  Further discussion of these and related properties can be found in \cite{Bergman}, and, in the context of topological groups, also \cite[Sect 4.E]{CH}.  
\medskip 

\noindent \textit{Strong boundedness and uncountable cofinality do not imply Schreier.}
This comes from the following example of a group with the strong boundedness property, due to Cornulier \cite{Cor}.
\begin{example} \label{finite ex}
Let $G$ be a finite, simple group, and let $H$ be the infinite direct product of countably many copies of $G$.   It is shown in \cite{Cor} that such a group $H$ is strongly bounded.  
We show that  $H$ does not have the Schreier property.  
Let $S = \{s_1, ..., s_k\}$ be a finite subset of $H$, and write $s_i = (s_{i,1}, s_{i,2}, ...)$ where $s_{i,j} \in G$.  Since $G$ is finite, there exists $g_1 \in G$ such that $s_{1, j} = g_1$ for infinitely many $j$.  Passing to a further infinite subset of indices, we can find $g_2 \in G$ such that $s_{2,j} = g_2$ and $s_{1, j} = g_1$ for all such $j$.  Similarly, one finds $g_1, g_2, ... g_k$ such that $s_{i, j} = g_i$ holds for each $i$ for infinitely many $j$.   Thus any word in the generators projects to the same element of $G$ in all of these infinitely many places.  In particular, a sequence such as $(g, \id, \id, \id ...), (\id, g, \id, \id, ...), (\id, \id, g, \id ...)$ where $g \neq \id \in G$, cannot be written as a word in $S$.  
\end{example}

Since every strongly bounded group has uncountable cofinality (see \cite[Remark 4.E.11]{CH}), Example \ref{finite ex} also gives an example of a group with uncountable cofinality that fails to have Schreier's property. 
\medskip

\noindent \textit{Strong distorsion implies strong boundedness.} 
Assume that $G$ is a strongly distorted group.  That $G$ has the Schreier property is immediate from the definition.  For strong boundedness, suppose for contradiction that $\ell$ is an unbounded length function on $G$.  Let $g_n$ be a sequence of elements in $G$ such that $\ell(g_n) > w_n^2$, where $w_n$ is the sequence given by the definition of strong distortion.  Then there is a finite set $S$ such that $g_n$ can be written as a word of length $w_n$ in $S$.  However, this implies that $\ell(g_n) \leq Kw_n$, where $K = \max \{ \ell(s) \mid s \in S\}$, giving a contradiction. 
\medskip

\noindent \textit{Schreier implies uncountable cofinality}
We show the contrapositive.  Suppose that $G_1 \subsetneq G_2 \subsetneq G_3 ... $ is an increasing union of subgroups with $\bigcup_n G_n = G$.  Choose $f_n \in G_n \setminus G_{n-1}$.  If $S \subset G$ is any finite set, then there is a maximum $i$ such that $S \cap G_i \neq \emptyset$, hence $S \subset G_i$ and does not generate $\{f_n\}$.

\paragraph{Contents and outline of paper.}
Section \ref{homeo-section} gives a direct proof of strong distortion for $\Homeo_0(\R)$, and therefore a quick answer to Schreier's question.   In Section \ref{diff-sec} we introduce further technical tools to prove Theorem \ref{schrier-thm} for closed manifolds.  The proofs of strong distortion of $\Diff^r(\R^n)$ and Theorem \ref{schrier-thm} are given in Sections \ref{Rn-sec} and \ref{noncompact-sec} respectively.

\begin{ackn}
The authors thank G. Bergman and Y. de Cornulier for comments, and Y.C. for pointing out Corollary \ref{LO cor}.  
K. Mann was partially supported by NSF award DMS-1606254.
\end{ackn}

\section{Strong distortion for $\Homeo(\R)$}  \label{homeo-section}

The purpose of this section is to give a quick answer to Schreier's question, and introduce some strategies to be used later in the proof of Theorem \ref{distorted-thm}.   Note that strong distortion is inherited from finite index subgroups, so it suffices to work with the index two subgroup of \emph{orientation-preserving} homeomorphisms of the interval, $\Homeo_0(\R)$.  

Given a generating set $S$ for a group $G$, word length of $g \in G$ with respect to $S$ is denoted $\ell_S(g)$.  

\begin{proposition}[Strong distortion for $\Homeo_0(\R)$] \label{homeo-R prop}
Given a sequence $\{f_n\} \subset \Homeo_0(\R)$, there exists a set $S \subset \Homeo_0(\R)$ with $|S| = 10$, such that $\ell_S(f_n) \leq 14n+12$ holds for all $n$.  
\end{proposition}

The first step in the proof is a simple factorization lemma for homeomorphisms.    Say that a set $X$ is a \emph{standard infinite union of intervals} if it is the image of $\bigcup_{n \in \Z} [n+\frac{1}{3}, n-\frac{1}{3}]$ under some $f \in \Homeo_0(\R)$.  We denote by $\supp(h)$ the support of a homeomorphism $h$.

\begin{lemma} \label{kgh-lemma}
Let $\{f_n\} \subset \Homeo_0(\R)$.   There exist sets $X$, $Y \subset \R$, each a standard infinite union of intervals, and for each $n$ a factorization $f_n= g_n h_n k_n$, 
where $k_n$ has compact support,  $\supp(g_n) \subset X$, and $\supp(h_n) \subset Y$. 
\end{lemma}  

\begin{proof}  
This is a special (easier) case of Lemma \ref{lemma.fragmentation.annuli} below, this case can be done by hand as follows.  We denote by $[a\pm \varepsilon]$ the interval $[a-\varepsilon, a+\varepsilon]$. 
First, we inductively define the endpoints of the intervals in $X$.   Assume without loss of generality that $f_0 = \id$, and let $X_0 = [-3, -1] \sqcup [1,3]$.  
Inductively, suppose we have defined $X_k = [- x_k^+, -x_k^-] \sqcup [x_k^-, x_k^+]$ and points $z_{k-1}$ for all $k < n$.  Let $x_{n}^- = x_{n-1}^+ + 1$. 
Choose $z_{n}$ large enough, so that the interval $[z_{n}\pm \frac{1}{2}]$ and all its images under $f_{0}, \dots, f_{n}$ are located on the right-hand side of $x_{n}^-$, and the interval $[-z_{n}\pm \frac{1}{2}]$ and all its images under $f_{0}, \dots, f_{n}$ are located on the left-hand side of $-x_{n}^-$.  For instance, one could take $z_n = \max\{(x_{n}^-, f_j^{-1}(x_{j}^-),  - f_j^{-1}(-x_{j}^-)  : j \leq n\}   + 1$.  
Now choose $x_{n}^+$ large enough so that $[-x_{n}^+, x_{n}^+]$ contains all the intervals $f_j\left( [-z_{n}\pm \frac{1}{2}] \right)$, $f_j \left( [-z_{n}\pm \frac{1}{2}] \right)$ for $j \leq n$. 
Let $X_{n}=[-x_{n}^+, -x_{n}^-] \cup [x_{n}^+, x_{n}^-]$.

The purpose of this construction is to guarantee that, for every $j<n$, there exists a homeomorphism of $\R$ supported on $X_{n}$ that agrees with $f_{j}$ on  $[-z_{n}\pm \frac{1}{2}] \cup [z_{n}\pm \frac{1}{2}]$.   Such a homeomorphism exists because $X_{n}$ contains an interval containing both $[-z_{n}\pm \frac{1}{2}]$ and its image under $f_j$ (and similarly for $[z_{n}\pm \frac{1}{2}]$ and its image).

Let $X := \sqcup X_j$, let
$Y'= \bigcup_{n \geq 1} [-z_n \pm\frac{1}{2}]  \cup [z_n \pm\frac{1}{2}]$, and let $Y = \R \setminus Y'$.  Then $X$ and $Y$ both  are standard infinite unions of intervals.
The observation in the previous paragraph says that, for each $n$, we can find $g_n \in \Homeo_0(\R)$ supported on $X$ that coincides with $f_n$ on the subset 
$$
\bigcup \limits_{m \geq n} \left[-z_m \pm \tfrac{1}{2}\right] \cup \left[z_m \pm \tfrac{1}{2}\right]
$$
 of $Y'$, so $g_n^{-1}f_n$ is the identity there. In particular $g_n^{-1}f_n$ fixes $\pm z_{n}$ and we may write $g_n^{-1}f_n = h_{n} k_{n}$ with $k_{n}$ supported on $[-z_{n}, z_{n}]$ and $h_{n}$ supported on the complement. Actually $h_{n}$ is supported on
 $$
 \R \setminus \left( [-z_{n}, z_{n}]  \cup \bigcup_{m \geq n} [-z_m \pm\frac{1}{2}]  \cup [z_m \pm\frac{1}{2}] \right)
 $$
 which is a subset of $Y$, and we have $f_{n} = g_{n} h_{n} k_{n}$ as required by the lemma.
 \end{proof} 

Now to prove Proposition \ref{homeo-R prop}, take the sequences $k_n$, $g_n$ and $h_n$ given by the lemma.  We will build sets $S_1$, $S_2$ and $S_3 \subset \Homeo_0(\R)$ with $|S_1| = 4$ and $|S_2| = |S_3| = 3$ such that $\ell_{S_1}(k_n) \leq 6n+4$, $\ell_{S_2}(g_n) \leq 4n+4$, and $\ell_{S_3}(h_n) \leq 4n+4$.

\begin{proof}
Given that each $k_n$ has compact support, we may take compact intervals $K_n$ with $\supp(k_n) \subset K_n$,  such that $K_i$ is contained in the interior of $K_{i+1}$, and such that $\bigcup_n K_n = \R$.  Let $d: \R \to \R$ be a homeomorphism such that $d(K_i)$ contains  $K_{i+1}$ for all $i$.  Then $\supp(d^n k_{n} d^{-n}) \subset K_1$.  

We now use a classical trick. It appears, perhaps first, in Fisher \cite{Fisher}, but also in a related form in \cite{Galvin} and later in $\cite{CF}$ (and probably elsewhere!).  

\begin{construction} \label{fisher-construction}
Suppose $\{a_n\}$ is a sequence of homeomorphisms supported on a set $Z$, and there exist homeomorphisms $T$ and $S$ such that 
\begin{enumerate}
\item the sets $Z$, $S(Z)$, $S^2(Z)$ $\dots$ are pairwise disjoint,
\item the sets $\supp(S)$, $T(\supp(S))$, $T^2(\supp(S)) \dots$ are pairwise disjoint, and
\item  The maximum diameter of the connected components of $T^n(\supp(S))$ and of $S^n(Z)$ converges to 0 as $n \to \infty$.
\end{enumerate}
Denote $a^b = bab^{-1}$. Since the map
$
a_{n}^{T^n S^m}
$
is supported on $T^n S^m(Z)$, the three above properties entail that the function 
$$A(x) := \prod_{n \geq 0, m \geq 0} a_{n}^{T^n S^m}(x)$$
defines a homeomorphism.  Moreover, it is easily verified that the commutator $$[A^{T^{-n}},S] = A^{T^{-n}} (A^{-1})^{ST^{-n}}   = a_{n}$$
 by checking this equality separately on each set $T^n S^m(Z)$.

\end{construction} 

\begin{remark} 
Other variants of condition 3 can also be used in this construction.  For example, it can be replaced by either of:
\begin{enumerate}[3']
\item The collection of sets $Z$, $S(Z)$, $S^2(Z)$ $\dots$ and $\supp(S)$, $T(\supp(S))$, $T^2(\supp(S)) \dots$ are locally finite.  
\end{enumerate}
\begin{enumerate}[3'']
\item The maximum diameter of a connected component of $S^n(Z)$ converges to 0, and  $\supp(S)$, $T(\supp(S))$, $T^2(\supp(S)) \dots$ is locally finite.  
\end{enumerate}

\end{remark}

We will apply Construction \ref{fisher-construction} to the sequence $a_n := d^n k_{n} d^{-n}$ supported on $K_1$.  To do this, we may take $T$ to be supported on a neighborhood of $K_1$, and to satisfy $T(K_1) \cap K_1 = \emptyset$.  Then let $S$ be a homeomorphism supported on a smaller neighborhood $N$ of $K_1$, small enough so that $T(N) \cap N = \emptyset$, and again satisfying $S(K_1) \cap K_1 = \emptyset$.   We can choose $T$ and $S$ such that property 3 of the construction is satisfied.
Let $S_1 = \{d, A, S, T\}$, then $k_n = d^{-n} a_n d^n$ is a word of length $6n+4$ in $S_1$

Similarly, given the sequence $\{g_n\}$ supported on $X$ (a standard union of disjoint intervals), we can find a homeomorphism $T'$ supported on a neighborhood $N_X$ of $X$ that consists of pairwise disjoint neighborhoods of the intervals comprising $X$, and satisfying $T'(X) \cap X = \emptyset$.  Then take $S'$ to be supported on a smaller neighborhood, say $N'_X$ of $X$, so that translates of $N'_X$ by $T'$ are also disjoint.  Choose $T'$ and $S'$ that satisfies property 3. Together with the construction, this gives a set $S_2$ of cardinality 3 so that each $g_n$ is a word of length $4n+4$ in $S_2$.  

Finally, as $Y$ is also a standard union of disjoint intervals, this same argument applies verbatim to the sequence $\{h_n\}$ supported on $Y$.  

\end{proof}

\begin{remark}
This proof can be generalized directly to $\Homeo_0(\R^n)$ using collections of disjoint concentric annuli in the place of our sets $X$ and $Y$ of disjoint intervals.   However, the strategy does not immediately apply to $\Diff^r_0(\R^n)$ for any $n$ and any $r \geq 1$, since the ``infinite product" of conjugates of compactly supported diffeomorphisms, as in Construction \ref{fisher-construction}, will not generally be differentiable.   
\end{remark}

We conclude this section by noting an interesting application to orderable groups.   

\begin{corollary} \label{LO cor}
Let $G$ be a countable left-ordered group.  Then there exists a finitely generated left-orderable group $H$ containing $G$. Moreover, one can order $H$ such that the inclusion $H \to G$ is order preserving. 
\end{corollary}

\begin{proof}
Given $G$, realize $G$ as a subgroup of $\Homeo_+(\R)$; this can be done so that the linear order on $G$ agrees with that on the orbit $G(0) \subset \R$ under the usual order on $\R$.  (This is standard, see eg. \cite[Prop. 1.1.8]{GOD}).  Proposition \ref{homeo-R prop} implies that $G \subset H$, for some finitely generated group $H \subset \Homeo_+(\R)$.  Now $H$ can be given a left-invariant order that agrees with the given order on $G$ -- in fact all of $\Homeo_+(\R)$ can be given such an order, following \cite[\S 1.1.3]{GOD}.
\end{proof}

\begin{remark} 
Related to order structures, we also note that the strategy of the proof of Proposition \ref{homeo-R prop} appears to give an alternative proof of results in \cite{DH}.   Droste and Holland show there that that the automorphism group of a doubly homogeneous chain (meaning a totally ordered set where the set of order-preserving bijections acts transitively on pairs) has uncountable cofinality.  Interpreting $[a, b]$ as $\{c : a \leq c \leq b\}$ in our proof allows one to extend it to a more general setting.  
\end{remark}

\section{Schreier's property for $\Diff^r_0(M)$, $M$ closed} \label{diff-sec}

In this section we prove Theorem \ref{schrier-thm} for the case of diffeomorphism groups of closed manifolds.  We defer the case of open manifolds until after the proof of strong boundedness for $\Diff^r_0(\R^n)$. 

The proof uses the following two classical results.

\begin{theorem}[Simplicity of diffeomorphism groups \cite{Anderson}, \cite{Mather1} \cite{Mather2}, \cite{Thurston}.]  \label{simplicity thm}
Let $M$ be a connected manifold (without boundary), and $r \neq \dim(M) +1$.  Then the identity component of the group of compactly supported $C^r$ diffeomorphisms of $M$, denoted $\Diff_c^r(M)$, is a simple group.
\end{theorem}

Here, the $C^\infty$ case is due to Thurston \cite{Thurston}, and the $C^r$ case, for $1 \leq r < \infty$ is from Mather \cite{Mather1, Mather2}.  Mather and Thurston's proofs use different arguments, but both deal with group homology and are quite deep.  The $C^0$ case of the theorem, modulo the next ``fragmentation lemma", is much easier and originally due to Anderson \cite{Anderson}.  

\begin{lemma}[Fragmentation] \label{frag lem}
Let $M$ be a compact (not necessarily closed) manifold, and $\mathcal{U}$ a finite open cover of $M$.  Then $\Diff_0^r(M)$ is generated by the set 
$$\{f \in \Diff_0^r(M) : \supp(f) \subset U \text{ for some } U \in \mathcal{U} \}.$$  
\end{lemma}

The proof of Lemma \ref{frag lem} for groups of homeomorphisms is a major result of Edwards and Kirby, it uses the topological torus trick \cite{EK}.  The proof for $C^r$--diffeomorphisms is much easier: it uses only the fact that each $C^r$ diffeomorphism near the identity can be written as the time one map of a time-dependent vector field; one then ``cuts off" such vector fields by suitable bump functions.   See \cite{Banyaga} or \cite{Bounemoura} for an exposition. 

We will also use a lemma on affine subgroups.   
\begin{lemma}[Existence of affine subgroups] \label{affine lem}
Let $B \subset \R^n$ be a compact ball.   There exist one-parameter families of smooth diffeomorphisms $f^t$ and $g^s$ supported on $B$ and satisfying the relations
$f^t g^s f^{-t} = g^{s e^t}$ for all $s, t$.  
\end{lemma}

The idea of the proof in the one-dimensional case is to conjugate the standard affine group in $\Diff^\infty_0(\R)$ generated by the flows $f^t(x) = e^t x$ and $g^s(x) = x + s$ by a suitable homeomorphism from $\R$ to $(0,1)$ so as to ``flatten" derivatives at the endpoints; this is generalized to higher-dimensional manifolds by embedding a family of copies of $(0,1)$ inside a ball. 

\begin{proof}
For the 1-dimensional case, we follow \cite[\S 4.3]{Navas}.  
Fix $\epsilon < \frac{1}{2}$, and define homeomorphisms $f_1: (0,1) \to \R$ and $f_2: (0,1) \to (0,1)$ by \\
$f_1(x) = \left\{ \begin{array}{rc} -1/x &  \text{ for } x \in (0, \epsilon) \\
e^{-1/x} & \text{ for } x \in (\epsilon, 1) \end{array} \right.$ \,\, 
$f_2(x) \left\{ \begin{array}{rc} 1/(1-x) &  \text{ for } x \in (0, \epsilon) \\
1 - e^{1/(x-1)} & \text{ for } x \in (\epsilon, 1) \end{array} \right.$ \\
and let $f:(0,1) \to \R$ be the composition $f = f_1 \circ f_2^2$.  

The standard affine group in $\Diff^\infty(\R)$ is has its Lie algebra generated by the vector fields $\frac{\partial}{\partial x}$ and $x \frac{\partial}{\partial x}$.  Thus,  $f^*(\frac{\partial}{\partial x})$ and $f^*(x \frac{\partial}{\partial x})$ generate an affine subgroup of $(0,1)$.   One checks that these extend to smooth vector fields on $[0,1]$ that are infinitely flat at the endpoints, hence extend to smooth vector fields on $(-\delta, 1+\delta)$ supported on $[0,1]$.  These generate an affine subgroup $G \subset \Diff^\infty([-\delta, 1+\delta])$ supported on $[0,1]$.   Let $G(n)$ be the affine subgroup of $\Diff^\infty([-\delta, 1+\delta] \times S^{n-1})$ given by the product action of $G$ on the $[-\delta, 1+\delta]$ factor, and trivial action on the $S^{n-1}$ factor.  

Finally, given a manifold $M$ of dimension $n$ and open ball $B$, we can take $\phi$ to be a smooth embedding of $(-\delta, 1+\delta) \times S^{n-1}$ in $M$, and consider the affine subgroup given by extending each element of $\phi G(n) \phi^{-1}$ to agree with the identity outside of the image of $\phi$.  
\end{proof}

Although Theorem \ref{simplicity thm} means that every $f \in \Diff^r_c(M)$ can be written as a product of commutators, Mather's proof is non-constructive, so gives no control on the norms of the elements in these commutators and the number of commutators in terms of the norm of $f$.  (It is however possible to control the norm and number in the $r = \infty$ and $r=0$ cases; see \cite{HRT} for the $C^\infty$ case, the $C^0$ case is an exercise.)  The benefit to working inside of affine subgroups is that elements close to the identity can always be written as commutators of elements close to the identity.   Precisely, we have the following corollary of Lemma \ref{affine lem}, giving control on norms of elements that will be used later on.    

\begin{corollary} \label{affine cor}
Let $r$ be arbitrary, and let $G$ be an affine subgroup of $\Diff^r_c(M)$ generated by $C^r$ flows $f^t$ and $g^s$ satisfying relations as in Lemma \ref{affine lem}.  For any neighborhood $\mathcal{U}$ of $\id$ in $\Diff^r_c(M)$, there exists a neighborhood $\mathcal{V}$ of $\id$ such that, if $g^s \in \mathcal{V}$, then $g^s$ can be written as a single commutator of elements of $\mathcal{U} \cap G$.  
\end{corollary}

\begin{proof}
Since the flows $f^t$ and $g^s$ are continuous in $t$ and $s$, it suffices to show that, given $\epsilon > 0$, there exists $\delta_0 > 0$ such that if $\delta < \delta_0$, then $g^{\delta}$ can be written as a commutator $[f^t, g^s]$ with $t, s < \epsilon$.  This is immediate from the relation in affine group, which gives $f^t g^s (f^t)^{-1} (g^s)^{-1}= g^{s (e^t -1)}$.  
\end{proof}

The next proposition is the main result of this section. 

\begin{proposition}[Theorem \ref{schrier-thm}, closed manifold case] \label{schrier-compact}
Let $M$ be a closed manifold, and $\{f_n\} \subset \Diff_0^r(M)$.  Assume $r \neq \dim(M)+1$.  Then there exists a finite set $S \subset \Diff_0^r(M)$ such that $\{f_n\} \subset \langle S \rangle$
\end{proposition}

We start with an obvious lemma.  
\begin{lemma}  \label{obv lemma}
Let $G$ be a group, and let $X$ be a generating set for $G$.  Then $G$ has the Schreier property if and only if, for every sequence $x_n \in X$, there exists a finite set ${\cal S} \subset G$ such that $\{x_n\} \subset \langle {\cal S} \rangle$.  
\end{lemma}

\begin{proof} Let $G$ be a group generated by a subset $X$. The condition on sequences in $X$ is an immediate consequence of the Schreier property.  For the converse, assume $X$ has the property in the lemma.  Now if $f_n$ is an arbitrary sequence in $G$, we may write $f_n = f_{n, 1}... f_{n, j(n)}$ where each $f_{n, i} \in X$.  Now apply the assumption from the lemma to the countable set $\{f_{n, i}\}$.  This provides a set that ${\cal S}$ that generates each $f_n$. 
\end{proof}  

Now to prove the proposition. 

\begin{proof}
Fix a finite cover of $M$ by open balls.  The fragmentation lemma states that the set of diffeomorphisms whose support lies in a single element of the cover is a generating set for $\Diff_0^r(M)$.  By Lemma \ref{obv lemma} and the fact that the cover is finite, it  suffices to show that for an open ball $B$ and any sequence $\{f_n\} \subset \Diff_c^r(B)$, there exists a finite set $S \subset \Diff^r(M)$ such that $\{f_n\} \subset \langle S \rangle$.  

Since $\Diff_c^r(B)$ is simple, Lemma \ref{affine lem} implies that it is generated by the set 
$$\{g : g \text{ is the time 1 map of a flow } g^s \text{ from an affine subgroup}\}.$$  
Thus, again using Lemma \ref{obv lemma}, we can reduce to the case where each $f_n$ is the time one map of a flow $g^s_n$ from some affine subgroup.  
 
 
The next short lemma is based on an idea of Avila \cite{Avila}.    To fix terminology, let $M$ be a $C^r$ manifold, and let $B'$ be an embedded Euclidean ball in $M$, i.e. the image of a standard Euclidean ball by some $C^r$ diffeomorphism $\phi$.   This allows us to push forward the standard $C^r$ norm on $\Diff_c^r(\R^n)$ to $\Diff_c(B')$, the subset of $\Diff_c(M)$ consisting of diffeomorphisms supported on the interior of $B'$.  Abusing notation somewhat, we denote this push-forward $C^r$ norm by $\| f \|_r$.   Note that the induced left-invariant distance $d_r(f, g) := \|f^{-1}g\|_r$ on $\Diff_c(B')$ generates the topology of $\Diff_c(B') \subset \Diff_c(M)$.

 \begin{lemma} \label{avila-lem}
 Let $Z \subset M$ be an open set, and $T \in \Diff^r_0(M)$ such that the translates $T^n(Z)$ are pairwise disjoint and contained in an embedded ball $B'$.   Then there exist $\epsilon_n \to 0$ (depending on $T$) such that, if $a_n$ is a sequence of diffeomorphisms with $\|a_n \| < \epsilon_n$ and support on $Z$, then the infinite product $\prod_n T^n a_n T^{-n}$ is a $C^r$ diffeomorphism.   
\end{lemma}

\begin{proof}
Fix $T \in \Diff^r_0(M)$ such that the translates $T^n(Z)$ are pairwise disjoint.   For each $n$, conjugation by $T^n$ is a continuous automorphism of $\Diff^r_0(M)$, so there exists $\epsilon_n$ such that, if $a_n$ has $C^r$-norm less than $\epsilon_n$, then $T^n a_n T^{-1}$ has $C^r$ norm less than $2^{-n}$.  Thus, for any such sequence $a_n$, the sequence 
$$
A_k := \prod \limits_{i=1}^k T^n a_n T^{-1}
$$
 is Cauchy, so converges in the $C^r$ topology to the diffeomorphism $\prod \limits_{n \in \N} T^n a_n T^{-n}$.  

\end{proof}

To apply this to our situation, let $Z \subset M$ be an open ball, and let $T$ and $S \in \Diff^r_0(M)$ be such that the translates $T^n(Z), S^m(Z)$ for $n \in \Z$ and $m \in \Z \setminus\{0\}$ are all pairwise disjoint.  If $\dim(M) = 1$, one can take $S$ and $T$ as in the proof of Proposition \ref{homeo-R prop}, the higher dimensional case is entirely analogous.   Using Lemma \ref{avila-lem}, let  $\epsilon_n$ be such that if $a_n$ and $b_n$ are sequences of diffeomorphisms with $\| a_n \|_r < \epsilon_n$ and $\| b_n \|_r < \epsilon_n$, then the infinite compositions  $\prod S^n a_n S^{-n}$ and $\prod T^n b_n T^{-n}$ are $C^r$ diffeomorphisms.  
By Corollary \ref{affine cor}, if we fix $k = k(n)$ sufficiently large, then we can write $g^{1/k}_n$ as a commutator $[a_n, b_n]$, such that the $\|a_n\|_r < \epsilon_n$ and $\|b_n\|_r < \epsilon_n$.   In this case, $g_n = [a_n, b_n]^{k(n)}$.

Now we apply Lemma \ref{avila-lem}.  Define $C^r$--diffeomorphisms $A$ and $B$ by 
$$A := \prod S^n a_n S^{-n}$$
$$B := \prod T^n b_n T^{-n}.$$
Note that the intersection of the supports of the maps  $S^{-n} A S^n$ and $T^{-n} B T^n$ is contained in $Z$, and on that set they coincide respectively with $a_{n}$ and $b_{n}$. Thus
 $[a_{n}, b_{n}] = [S^{-n} A S^n, T^{-n} B T^n]$
which shows that the sequence $g_n$ is generated by the set ${\cal S} = \{A, B, T, S\}$.   This completes the proof.  

\end{proof}

\subsection{Mapping class groups, extensions and quotients}  \label{MCG-sec}

To finish the proof of Theorem \ref{schrier-thm} for closed manifolds, we need to show $\Diff^r(M)$ has the Schreier property if and only if the mapping class group is finitely generated.  This is a direct consequence of the following observation.  

\begin{proposition}
If $G$ is a group with the Schreier property, then any quotient of $G$ has the Schreier property.  
If $A$ and $C$ are groups with the Schreier property, then any extension 
$1 \to A \to B \to C \to 1$ has the Schreier property.   

The same statements hold when the Schreier property is replaced by strong distortion.  
\end{proposition}

\begin{proof} 
The first statement is immediate from the definition of the property.  For the second statement, given a sequence $b_n \in B$, let $S_1 \subset C$ be a finite set generating the images of $b_n$ in $C$, and let $S'_1$ be a transversal for $S_{1}$ in $B$.   Then, for each $n$ there exists $a_n \in A$ such that $a_n b_n \in \langle S'_1 \rangle$.  Let $S_2 \subset A \subset B$ be a finite set generating $\{a_n\}$, and let $S = S'_1 \cup S_2$.  

In the case where $A$ and $C$ have strong distortion (say with sequences $w_n^A$ and $w_n^C$, respectively), choosing $S'_1$ such that the images of $b_n$ in $C$ have length at most $w_n^C$ in $S_1$, and $S_2$ such that $a_n$ has length at most $w_n^A$ in $S_2$, shows that $B$ is strongly distorted with sequence $w_n^A +w_n^C$. 
\end{proof}

Now our claim about mapping class groups follows from the fact that a countable group has the Schreier property if and only if it is finitely generated, and that the mapping class group is the quotient of $\Diff^r(M)$ by $\Diff^r_0(M)$.  
 \qed

Note that examples where mapping class groups cannot be finitely generated do indeed occur: for one concrete example, Hatcher \cite{Hatcher} and Hsiang--Sharpe \cite{HS} have independently computed the mapping class group $\Diff^\infty(\mathbb{T}^5)/\Diff^\infty_0(\mathbb{T}^5) $, and it is not finitely generated.

\section{Strong distortion for $\Diff^r_0(\R^n)$}  \label{Rn-sec}

In this section we will prove the following result.

\begin{theorem}\label{theo.diff.R} (Strong distortion for $\Diff^r_{0}(\R^d)$)
Let $0 \leq r \leq \infty, r \neq d+1$, and let $\{f_n\} \subset \Diff^r_{0}(\R^d)$.  Then there exists a set ${\cal S} \subset \Diff^r_{0}(\R^d)$ with 17 elements, such that each $f_n$ can be written as a word of length $50n+24$ in ${\cal S}$.  
\end{theorem}

Since $\Diff^r_0(\R^d)$ is the index two subgroup of orientation preserving $C^r$ diffeomorphisms in $\Diff^r(\R^d)$, an argument as in Section \ref{MCG-sec} implies that $\Diff^r(\R^d)$ is strongly distorted also.

To prove Theorem \ref{theo.diff.R}, we additionally need the following theorem of Burago, Ivanov, and Polterovich.  

\begin{theorem*}[Theorem 1.18 in \cite{BIP}]
Let $M$ be a manifold diffeomorphic to a product $M' \times \R^{n-1}$.  If $\Diff^r_c(M)$ is perfect, then any element may be written as the product of two commutators.
\end{theorem*}

This theorem applies in the more general context where $M$ is a ``portable manifold", but we only need this special case here.  The statement in \cite{BIP} is given for $C^\infty$ diffeomorphisms, but the proof does not use smoothness and applies directly to the $C^r$ case, for any $r$.  

The uniform bound on commutator length from Burago--Ivanov--Polterovich will help us control word length in the proof of strong boundedness.  The other major tool towards this end is a variant of  Lemma \ref{avila-lem} avoiding the earlier hypothesis that the norms of diffeomorphisms $a_n$ are bounded by a sequence tending to zero.  Instead, we will use the unboundedness of $\R^d$ to displace supports so as to avoid accumulation points.  This is the purpose of the next technical lemma.  

\begin{lemma}\label{lemma.avila.sets}
There exist $S,T \in \Diff_{0}^\infty(\R)$ that are the identity on $(-\infty, 0]$, and a sequence $\{I_{k}\}_{k \geq 0}$ of intervals in $(0,+\infty)$, such that

\begin{enumerate}
\item the family
$
\{S^i I_{k_{1}}, T^j I_{k_{2}} , \ \ i, j, k_{1}, k_{2} \geq 0 \}
$
is locally finite, and
\item the intervals 
$S^i I_{k_{1}}, T^j I_{k_{2}}$ for $ i, j \in \Z$ and $k_{1}, k_{2} \geq 0$
are pairwise disjoint (with the trivial exception of $S^0 I_{k} = T^0_1 I_{k} = I_k$ for all $k$).
\end{enumerate}
\end{lemma}

Figure \ref{avila-fig} gives a graphical description of properties 1 and 2 from Lemma \ref{lemma.avila.sets}.  The figure shows a configuration of rectangles $I_k$ in $\R^2$, and their images under diffeomorphisms $S$ and $T$, that satisfy both properties.  It is much harder to achieve this configuration for intervals in $\R$; this is the technical work in proof of the lemma.  

\begin{figure}[ht]
 \labellist 
  \scriptsize \hair 2pt
   \pinlabel $I_0$ at 140 175
   \pinlabel $S$ at  175 215
   \pinlabel $\nearrow$ at 178 200
   \pinlabel $S(I_0)$ at 205 200
   \pinlabel $T$ at 175 133
   \pinlabel $\searrow$ at 178 150
   \pinlabel $T(I_0)$ at 205 145
   \pinlabel $\rightarrow$ at 235 210
   \pinlabel $\rightarrow$ at 235 138
     \pinlabel $\rightarrow$ at 296 210
   \pinlabel $\rightarrow$ at 296 138
   \pinlabel $\nearrow$ at 108 150
   \pinlabel $\searrow$ at 108 195
   \pinlabel $S$ at  110 210
    \pinlabel $T$ at 110 135
      \pinlabel $\rightarrow$ at 60 200
   \pinlabel $\rightarrow$ at 60 145
   
   \pinlabel $S^2(I_0)$ at 265 210
   \pinlabel $T^2(I_0)$ at 265 138
   \pinlabel $I_1$ at 140 50
     \pinlabel $S$ at  175 85
   \pinlabel $\nearrow$ at 178 70
    \pinlabel $T$ at 175 10
   \pinlabel $\searrow$ at 178 25
    \pinlabel $\rightarrow$ at 235 85
   \pinlabel $\rightarrow$ at 235 15
       \pinlabel $\rightarrow$ at 296 85
   \pinlabel $\rightarrow$ at 296 15
      \pinlabel $\nearrow$ at 108 30
   \pinlabel $\searrow$ at 108 70
   \pinlabel $S$ at  110 83
    \pinlabel $T$ at 110 15
    \pinlabel $\rightarrow$ at 60 75
   \pinlabel $\rightarrow$ at 60 20
   
   \pinlabel $\dots$ at 360 210
   \pinlabel $\dots$ at 360 138
   \pinlabel $\dots$ at 360 85
   \pinlabel $\dots$ at 360 15   
   \pinlabel $\vdots$ at 140 0
     \endlabellist
  \centerline{
    \mbox{\includegraphics[width=3.6in]{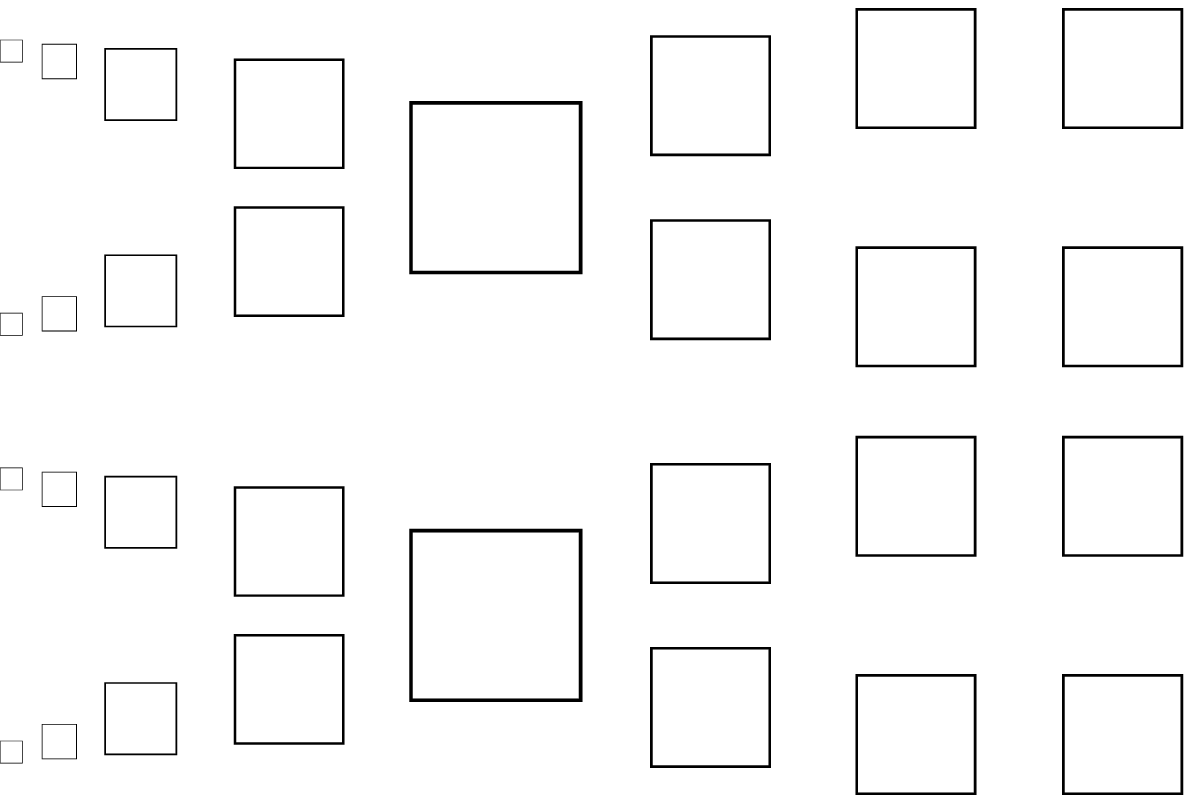}}
    \vspace{10pt} }
\caption{\small Configuration of cubes in $\R^2$ satisfying the properties of Lemma \ref{lemma.avila.sets}}
\label{avila-fig}
\end{figure}

\begin{proof}
Let $S$ be a smooth diffeomorphism of the line which is the identity on $(-\infty, 0]$, and which coincides with an affine map fixing $2$, say $x \mapsto 2(x-2)+2$, on $[2,+\infty)$.
Similarly, let $T_{0}$ be a smooth diffeomorphism of the line which is the identity on $(-\infty, 0]$, which coincides with $S$ on $[3,+\infty)$, fixes $1$, and has no fixed point in $(1,+\infty)$.  
Note that for every point $x > 2$, the sequence $S^{-n}(x)$ converges to $2$ as $n \to \infty$, while the sequence $T^{-n}_{0}(x)$ converges to $1$.  

We will define the intervals $I_0$, $I_1$, $I_2$,... iteratively, modifying $T_{0}$ at each step to produce diffeomorphisms $T'_0$, $T'_1$, $T'_2$,... , designed to converge to a diffeomorphism $T$ with our desired properties.  

Take any point $x_{0} \geq 3$ such that $x_0 \notin \{T_0^{k}(2) : k>0\}$.  $I_0$ will be the closure of a small neighborhood of $x_0$, of size to be determined after the construction of $T'_0$.  For this, we consider the backward iterates of $x_0$ under $T_0$ and $S$.  If there is no common iterate, i.e. $\{S^{-n} (x_{0}) : n > 0 \} \cap \{T_0^{-n} (x_{0}): n > 0 \} = \emptyset$,  then we let $T'_{0} = T_{0}$.  Otherwise, we modify $T_{0}$ as follows. Choose $x'_{0}$  close to $x_{0}$ outside the countable set $\{T_{0}^m S^{-n} (x_{0}) :  n, m>0\}$, so that the backward iterates of $x'_{0}$ under $T_{0}$ are disjoint from the backward iterates of $x_{0}$ under $S$. Then modify $T_{0}$ near $T_{0}^{-1} (x'_{0})$ to obtain a map $T'_{0}$ such that $T'_{0} (T_{0}^{-1} (x'_{0})) = x_{0}$.  
This can be done, for instance, by taking a diffeomorphism $h$ with support on a small neighborhood of $T_{0}^{-1} (x'_{0})$ 
disjoint from $\{T_{0}^{-m} (x_{0}) : m>1\}$, 
and setting $T'_0 = h \circ T_0$.

Now the backward iterates of $x_{0}$ under ${T'}_{0}$ coincide with the backward iterates of $x'_{0}$ under $T_{0}$,
and thus are disjoint from the backward iterates of $x_{0}$ under $S$.  
Note that by choosing $x'_{0}$ sufficiently close to $x_{0}$ we may keep the property that $x_{0} \notin \{{T'_0}^{k}(2) : k>0\}$.  

Since $\{ {T'_{0}}^{-n}(x_{0}) : n > 0\} \cap [2, x_{0}]$ 
is finite and does not contain $2$, and since $[2, x_{0}]$ contains $\{S^{-n} (x_{0}) : n > 0 \}$, if $I_{0}$ is a sufficiently small neighborhood of $x_0$, then every image ${T'_{0}}^{-n}(I_0)$ will be disjoint from $\bigcup_{n>0} S^{-n}(I_0)$.  Fix any such interval $I_0$.  

At this point, the $T'_{0}$-forward iterates $\{ {T'_{0}}^{n}(I_0) : n > 0\}$ of $I_{0}$ coincide with its $S$-forward iterates.  
We now further modify $T'_{0}$ so that they are pairwise disjoint from the iterates under $S$.  To do this, fix a small neighborhood $U$ of $S(I_{0})$ 
so that $U \cap S(U)= \emptyset$,
and let $I'_0$ be a small interval in $U$ 
disjoint from $S(I_{0})$.
Then all the $S$-forward iterates of $I'_{0}$ and $S(I_{0})$ are disjoint. Now modify $T'_{0}$  by postcomposing it with a diffeomorphism $h$ supported on $U$ and such that $h(I_0) = I'_0$. Call this new map $T_1$, and note that $T_1(I_0) = I'_0$.  We have achieved the following properties:

\begin{enumerate}[i)]
\item the family
$
\{S^i I_{0}, T^j_1 I_{0} : \ \ i, j  \geq 0 \}
$
is locally finite,
\item the intervals
$S^i I_{0}, T^j_1 I_{0}$ for $i, j \in \Z$
are pairwise disjoint, with the trivial exception of $S^0 I_{0} = T^0_1 I_{0} = I_0$.  
\end{enumerate}

\bigskip

Let $Z_0$ be the union of the intervals in the family from ii) above.
We define the interval $I_{1}$ by a similar procedure to that of $I_0$. Choose some point $x_{1} > S(x_{0})$, outside $Z_0$, which is not a forward iterate of the point $2$ under $T_{1}$. As before, modify $T_{1}$ near $T_{1}^{-1}(x_{1})$ if necessary to obtain a map $T'_{1}$ so that the set of backward iterates of $x_{1}$ under $T'_{1}$ is disjoint from the set of backward iterates of $x_{1}$ under $S$.  The same argument as above implies that we may find a small interval $I_{1}$ around $x_{1}$, taken sufficiently small so that it is disjoint from the set $Z_0$, such that every  $T'_{1}$-backward iterate of $I_{1}$ is disjoint from every $S$-backward iterate of $I_{1}$.  As the forward iterates of $I_{1}$ under $T'_{1}$ and under $S$ coincide we now modify $T'_{1}$ in a neighborhood of $I_{1}$, to get a map $T_{2}$ with the property that all the $T_{2}$-forward iterates of $I_{1}$ are disjoint and disjoint from its $S$-forward iterates.

We repeat the same process iteratively.  At the $k^{th}$ step, choose $x_k > S(x_{k-1})$, modify the already defined $T_k$ to $T'_{k}$ as above in order to be able to choose a suitable small neighborhood $I_k$ of $x_k$ and then modify $T'_k$ by composing with a diffeomorphism supported on a neighborhood of $S(I_k)$ to get $T_{k+1}$ so that the following properties hold: 
\begin{enumerate}[i)]
\item the family
$
\{S^i I_{m}, T_{k+1}^j I_{m} : \ \ i, j  \geq 0, m \leq k \}    
$
is locally finite, and
\item the intervals in the family
$
S^i I_{m}, T_{k+1}^j I_{m}$ for  $i, j \in \Z, m \leq k
$
are pairwise disjoint, with the trivial exception $S^0 I_{m} = T^0 I_{m}$.  
\end{enumerate}

Since, at each step, we choose $I_{k}$ to be a small interval about a point $x_k \geq S(x_{k-1})$, the sequence of intervals $\{I_{k}\}$ is locally finite.  And since on every compact subset $K$ of the line, all but a finite number of the maps $T_{k}$ agree, the sequence $\{T_{k}\}$ converges to an element $T$ of $\Diff^\infty(\R)$.   
By construction, these maps $T,S$ and the sequence 
$\{I_{k}\}$ satisfy properties 1. and 2. from the statement of the lemma.
\end{proof}

The next step is a natural generalization of Lemma \ref{kgh-lemma}.  However, since we are now working in higher dimensions, we need to use the annulus theorem (proved by Kirby \cite{Kirby} and Quinn \cite{Quinn} for the difficult case $r=0$).  As an alternative to the annulus theorem, one can use the related Edwards--Kirby theory of deformations of embeddings.  We will take this latter approach in the next section, for now we use the more familiar annulus theorem directly.   The precise consequence that we need is the following.  

\begin{lemma}[consequence of the annulus theorem] \label{annulus lem}
Let $B_1 \subset B_2 \subset B_3 \subset B_4$ be standard Euclidean closed balls in $\R^d$ centered at $0$ with pairwise disjoint boundaries.  Let $A$ be the annulus $B_3 \setminus \interior(B_2)$.  
Suppose $f \in \Diff^r_0(\R^d)$ satisfies $f(A) \subset \interior (B_4) \setminus B_1$, and that $f(A)$ is homotopically essential in the annulus $\interior (B_4) \setminus B_1$.  Then there exists $h \in \Diff^r_0(\R^d)$ supported on $B_4 \setminus B_1$ that agrees with $f$ on $A$. 
\end{lemma}

\begin{proof}
Let $B(R)$ denote the standard Euclidean ball of radius $R$.   It is a standard corollary of the annulus theorem that, 
if $\gamma$ is a $C^r$ embedding of $B(\frac{1}{2})$ into $B(1)$, then $B(1) \setminus \interior(\gamma(B(\frac{1}{2})))$ is $C^r$-diffeomorphic  to $B(1) \setminus \interior(B(\frac{1}{2}))$.  Moreover, the diffeomorphism can be taken to agree (meaning to agree up to order $r$) with the identity on $\partial B$ and agree with $\gamma$ on $\partial B(\frac{1}{2})$.    

This means that, given $f$ as in the lemma, we may find $h_1: B_4 \setminus B_3 \to B_4 \setminus f(B_3)$ that is the identity on $\partial B_4$ and agrees with $f$ on $\partial B_3$.  Extend $h_1$ to a homeomorphism of $\R^d$ that agrees with $f$ on $B_3$ and the identity outside of $B_4$.   By the same argument, we may find $h_2$ that agrees with the identity on $f(\partial B_2)$ and agrees with $f^{-1}$ on $f(\partial B_1)$; extend $h_2$ to be the identity outside of $f(B_2)$ and agree with $f^{-1}$ on $f(B_1)$.  Now $h:= h_2 h_1$ is supported on $B_4 \setminus B_1$ and agrees with $f$ on $A$.  
\end{proof}

\begin{lemma}\label{lemma.fragmentation.annuli}
Let $\{f_n\} \subset \Diff_0(\R^d)$.  
There exists sets $X$ and $Y$, each a union of a locally finite family of disjoint concentric annuli, such that we can write each element $f_n$ as a product $f_n = k_n g_n h_n$, where each $k_n$ has compact support, $\supp(g_n) \subset X$, and $\supp(h_n) \subset Y$.
\end{lemma} 

\begin{proof}
Similarly to the proof of Lemma \ref{kgh-lemma}, we first construct two sequences of concentric annuli.  For $R> 0$, let $B(R)$ denote the closed ball of radius $R$ about 0 in $\R^d$.  
The annuli will be defined by 
$$
A_{N} = B(R_{N}^+) \setminus \interior B(R_{N}^-), \ \ \  A'_{N} = B({R'}^+_{N}) \setminus \interior B({R'}^-_{N}) \ \ \ (N\geq0)
$$
and have the properties that 
\begin{itemize}
\item the annuli $A'_{N}, N\geq0$ are pairwise disjoint,
\item for every $N\geq0$, $A_{N}$ is contained in $A'_{N}$, 
\item for every $N\geq0$ and for every $n \leq N$, 
\begin{itemize}
\item  $B({R'}^-_{N})$ is contained in the interior of $f_{n}(B(R^-_{N})$,
\item $f_{n}(B(R^+_{N}))$ is contained in the interior of $B({R'}^+_{N})$.
\end{itemize}
\end{itemize}
Note that the last point says that $f_{n}(A_{N})$ is contained in $A'_{N}$ in a homotopically essential way.

We construct these annuli by induction, the procedure is quite analogous to that in Lemma \ref{kgh-lemma}.  First set ${R'}^-_{0}=1$, then choose $R^-_{0}$ large enough so that the ball $B({R'}^-_{0})$ is contained in the interior of $f_{0}(B(R^-_{0}))$, then choose for $R^+_{0}$ any number larger that $R^-_{0}+1$, and finally choose ${R'}^+_{0}$ large enough so that $f_{0}(B(R^+_{0}))$ is contained in the interior of  $B({R'}^+_{0})$.
Now assume that the annuli have been constructed up to step $N$, satisfying the above properties.
We construct $A_{N+1}$ and $A'_{N+1}$ as follows. 
First choose ${R'}^-_{N+1}$ greater than ${R'}^+_{N}$.
Then choose $R^-_{N+1}$ large enough so that for every $n=0, \dots , N+1$,  the ball $B({R'}^-_{N+1})$ is contained in  the interior of $f_{n}(B(R^-_{N+1}))$.
 Then choose for $R^+_{N+1}$ any number larger that $R^-_{N+1}+1$.
 Finally choose ${R'}^+_{N+1}$ large enough so that for every $n=0, \dots , N+1$, the set $f_{n}(B(R^+_{N+1}))$ is contained in  the interior of $B({R'}^+_{N+1})$.

Now let us fix some $n \geq 0$, and define the maps $k_{n}, g_{n}$ and $h_{n}$ as follows.   The property that, for any $N \geq n$, the annulus $f_{n}(A_{N})$ is contained in $A'_{N}$ in a homotopically essential way
means that we can use Lemma \ref{annulus lem} to find $h_n \in \Diff^r(\R^d)$ supported in the disjoint union $X := \cup_N A'_{N}$, and that coincides with $f_{n}$ on a neighborhood of each $A_{N}$ with $N \geq n$.  Fix such an $h_n$.   Let $k_{n}$ agree with $f_{n} h_{n}^{-1}$ on the ball $B(R^-_{n})$, and be the identity elsewhere.  Define  $g_{n}$ to be the restriction of $f_{n} h_{n}^{-1}$ to the complement of this ball, and the identity elsewhere. Note that  $f_{n} = k_{n} g_{n} h_{n}$, and that $g_{n}$ is compactly supported in the disjoint union of annuli
$$
Y := \bigcup_{N \geq n} B(R^-_{N+1}) \setminus \interior B(R^+_{N}), 
$$
this proves the lemma. 
\end{proof}

\begin{proof}[Proof of Theorem~\ref{theo.diff.R}]
Let $\{f_{n}\}$ be a sequence in $\Diff_{0}^r(\R^d)$. 
We first apply Lemma \ref{lemma.fragmentation.annuli}, to get two sets $X, Y$ and for each $n$ a decomposition $f_{n} = k_n g_n h_n$, with $\supp(k_n)$ compact, $\supp(g_n) \subset X$ and $\supp(h_n) \subset Y$. 

We first take care of the sequence $\{g_{n}\}$ supported in $X$. Apply Lemma \ref{lemma.avila.sets} to get maps $S,T \in \Diff_{0}^\infty(\R)$ and a sequence $\{I_{k}\}$ of intervals in $(0,+\infty)$. Using polar coordinates, we identify $\R^d \setminus \{0\}$ with $\R \times \mathbb{S}^{d-1}$, and let 
$$
\hat I_{k} = I_{k} \times \mathbb{S}^{d-1}, \ \ \hat S = S \times \mathrm{Id},  \ \ \hat T = T \times \mathrm{Id}.
$$
Note that since $S$ and $T$ are the identity near $-\infty$, the maps $\hat S, \hat T$ extends to smooth diffeomorphisms of $\R^d$ fixing $0$. Also note that properties 1 and 2 of Lemma~\ref{lemma.avila.sets} still hold if we replace 
$\{I_{k}\}$, $S$ and $T$ by the sequence of annuli $\{\hat I_{k}\}$ and the maps $\hat S, \hat T$.

Since $\{\hat I_{k}\}$ is a locally finite sequence of concentric pairwise disjoint annuli, there exists a diffeomorphism that sends the union of the $\hat I_{k}$'s onto a neighborhood of the set $X$. Up to conjugating by this diffeomorphism, we may assume that $X = \cup_{k \geq 0} \hat I_{k}$, and each $g_{n}$ is supported in the interior of $X$.

We now appeal to Burago-Ivanov-Polterovich's theorem stated above: for each fixed $n$ and $k$ we may write the restriction of $g_{n}$ to  $\hat I_{k}$ as a product of two commutators of diffeomorphisms supported in $\hat I_{k}$. Since the $\hat I_{k}$ are pairwise disjoint, we may take composition over $k$ and get $C^r$ diffeomorphisms  $a_n, b_n, a'_n$ and $b'_n$ supported in the union of the $\hat I_{k}$, such that  $g_n = [a_{n}, b_{n}] [a'_{n}, b'_{n}]$.

We work first with the sequence $\{a_n\}$ and $\{b_n\}$ applying the same strategy from the compact manifold case.  

Let 
$$
A = \prod_{n \geq 0} \hat S ^n a_{n} \hat S ^{-n}, \ \ \ B = \prod_{n \geq 0} \hat T ^n b_{n} \hat T ^{-n}.
$$
Note that these infinite products define \emph{diffeomorphisms} of $\R^d$, because of local finiteness of the supports (property 1 of Lemma~\ref{lemma.avila.sets}.)  Now Property 2 of the same lemma implies that for every $n \geq 0$ we have
$$
[a_{n}, b_{n}] = [\hat S^{-n} A \hat S^n, \hat T^{-n} B \hat T^n].
$$

The same strategy (and the same $\hat S$ and $\hat T$) can be used to give $A'$ and $B'$ such that $[a'_{n}, b'_{n}] = [\hat S^{-n} A' \hat S^n, \hat T^{-n} B' \hat T^n]$.

We have just shown that any sequence $\{g_n\}$ supported in $X$ can be written as a word in $\{\hat S, \hat T, A,B,A',B' \}$ of length $2(4(2n+1))$.   We can do the same for the sequence $\{f_n\}$ supported in $Y$, writing each as a word of length $16n+8$ in a set of 6 different elements, say $\{\hat S_2, \hat T_2, A_2,B_2,A'_2, B'_2 \}$.  It remains only to treat the sequence $\{k_{n}\}$.
Let $B(r_n)$ be a sequence of nested balls of increasing radii such that $\supp(k_n) \subset B(r_n)$.  Fix a ball $K_0 \subset X$, and let $\phi \in \Diff_{0}^r(\R^d)$ be a diffeomorphism such that, for every $n \geq 0$, we have $\phi^n(B(r_n)) \subset K_0$.  Then $\phi^{-n} k_{n} \phi^n$ is supported in $K_0 \subset X$, so the same argument for the sequence $\{g_n\}$ applies to $\{\phi^{-n} k_{n} \phi^n\}$; in fact, we may even use the same diffeomorphisms $\hat S$ and  $\hat T$.  This gives a set $\{\hat S, \hat T, A_3,B_3,A'_3, B'_3 \}$ so that each $\phi^n k_n \phi^{-n}$ can be written as a word of length $16n+8$. 

Thus, taking $\mathcal{S} : = \{\phi, \hat S, \hat T, \hat S_2, \hat T_2, A, B, A', B', A_i,B_i,A'_i, B'_i : i = 2, 3 \}$ as a generating set, $\ell_{\mathcal{S}}(k_n) \leq 18n+8$.  Combined with the estimates above, this gives 
$\ell_{\mathcal{S}}(f_{n}) \leq 50n+24$.
This completes the proof.
\end{proof}

\section{The Schreier property for $\Diff^r_0(M)$, $M$ noncompact} \label{noncompact-sec}

This short section gives the necessarily generalizations to Theorem \ref{theo.diff.R} in order to prove the following.  

\begin{proposition} \label{open-M-schrier}
Let $M$ be an open manifold diffeomorphic to the interior of a compact manifold with boundary.  Then $\Diff_0(M)$ has the Schreier property.  

In the special case that $M \cong N \times \R^k$ for some compact manifold $N$, then $\Diff_0(M)$ is also strongly distorted.  
\end{proposition}

The proof of this proposition follows the same strategy as the $\R^n$ case, but in place of the annulus theorem, we use the following related result (which is a difficult theorem in the $C^0$ case).    Recall that the \emph{trace} of an isotopy $f^t$, $t \in [0,1]$ of a set $C$ is defined to be $\bigcup_{t \in [0,1]} f^t(C)$

\begin{lemma} \label{isotopy-lem}
Let $f \in \Diff^r_0(M)$, let $f^t$ be an isotopy from $\id = f^0$ to $f = f^1$, and let $C \subset M$ be a compact set.  Given a neighborhood $U$ of the trace of $C$ under $f^t$, there exists $g \in \Diff^r_0(M)$ supported on $U$ and agreeing with $f$ on $C$.  
\end{lemma}

\begin{proof}
The $C^0$ case follows from the embedding theory of Edwards and Kirby, this statement is exactly the generalization of \cite[Cor. 1.2]{EK} explained in the second remark of \cite[p. 79]{EK}.   The case for $r \geq 1$ is easy:  one thinks of $\frac{\partial}{\partial t}f^t$ as defining a time-dependent vector field $X_t$ on $M$.  One then cuts off $X_t$ using a bump function that is identically one on the trace, and vanishes outside $U$.  The time one map of the resulting time-dependent vector field is the desired diffeomorphism $g$.  
\end{proof}

\begin{proof}[Proof of Proposition \ref{open-M-schrier}]
Let $M$ be an open manifold diffeomorphic to the interior of a compact manifold with boundary.  Then $\partial M$ is a compact (possibly disconnected) $n-1$ dimensional manifold, and a neighborhood of the union of ends of $M$ is diffeomorphic to $\partial M \times \R$.    

Let $\{f_n\}$ be a sequence in $\Diff_0(M)$.  
We will use Lemma \ref{isotopy-lem} to write $f_n$ as a product $k_n g_n h_n$, where $k_n$ has compact support, and $g_n$ and $h_n$ are supported in the union of ends of $M$.  Moreover, we will have that $g_n$  is supported in a set $X$ diffeomorphic to $\partial M \times \bigcup_{n > 0} [n+\tfrac{1}{3}, n-\tfrac{1}{3}]$, and $h_n$ is supported in a set $Y$ of the same form.   After this, the proof will proceed much as before, with $X$ and $Y$ playing the roles of the unions of annuli from the $M = \R^n$ case.  

To produce $g_n$ and $h_n$, fix an identification of the complement of a compact set in $M$ with $\R \times \partial M$, and fix isotopies $f^t_n$ from $f_n$ to $\id$.   Imitating notation from the previous proof, for $R > 0$, let $B(R):= (-\infty, R] \times \partial M \subset \R \times \partial M \subset M$.  
We next construct sequences $R^{\pm}_N$, ${R'}^{\pm}_N$.  

Set ${R'}^-_{0}=1$, then choose $R^-_{0}$ large enough so that the ball $B({R'}^-_{0})$ is contained in the interior of $\bigcup_t f_{0}^t(B(R^-_{0}))$.  Now choose $R^+_{0}$ to be any number larger that $R^-_{0}+1$, and finally choose ${R'}^+_{0}$ large enough so that $\bigcup_t f^t_{0}(B(R^+_{0}))$ is contained in the interior of  $B({R'}^+_{0})$.   The construction of  $R^{\pm}_n$ and ${R'}^{\pm}_n$ is by the same inductive procedure as the $\R^n$ case, except that we require $R^-_{N+1}$ to be large enough so that for every $n=0, \dots , N+1$,  the ball $B({R'}^-_{N+1})$ is contained in  the interior of the trace $\bigcup_t f^t_{n}(B(R^-_{N+1}))$, and ${R'}^+_{N+1}$ to be large enough so that for every $n=0, \dots , N+1$, we have $\bigcup_t  f^t_{n}(B(R^+_{N+1}))$ contained in  the interior of $B({R'}^+_{N+1})$.

Let $A_N =  B(R_{N}^+) \setminus \interior B(R_{N}^-)$ and $A'_{N} = B({R'}^+_{N}) \setminus \interior B({R'}^-_{N})$, for $N\geq0$.  
Now Lemma \ref{isotopy-lem} implies that there exists $h_n \in \Diff^r_0(M)$ supported in $X := \cup_N A'_{N}$, and coinciding with $f_{n}$ on a neighborhood of each $A_{N}$ with $N \geq n$.  Fix such an $h_n$.   Let $k_{n}$ agree with $f_{n} h_{n}^{-1}$ on the union of $B(R^-_{n})$ with the compact part (the complement of the ends) of $M$, and be the identity elsewhere.  Define  $g_{n}$ to be the restriction of $f_{n} h_{n}^{-1}$ to the complement of this ball, and the identity elsewhere.  As before, $f_{n} = k_{n} g_{n} h_{n}$, and $g_{n}$ is compactly supported in the disjoint union
$$
Y := \bigcup_{N \geq n} B(R^-_{N+1}) \setminus \interior B(R^+_{N}), 
$$

Following the proof of the $M = \R^n$ case verbatim, but replacing $\mathbb{S}^{d-1}$ with $\partial M$, we conclude that $\{g_n\}$ and $\{h_n\}$ can each be written as words of length $16n+8$  in sets of $6$ elements.  In the special case $M \cong \R^k \times N$, then $\supp(k_n)$ is contained in a set of the form $K_n \times N$, where $K_n$ is a compact set in $\R^k$.  Moreover, in this case, we have $A'_n \cong S^k \times N$.  Analogous to the $\R^n$ case, one can therefore find a diffeomorphism $\phi$ such that $\phi^n(K_n \times N) \subset A'_0 \subset X$.   Thus, the previous argument shows that $k_n$ can be written as a word of length $16n+8$ in a finite set; showing that $\Diff^r_0(M)$ is strongly distorted.  

In the general case, $\supp(k_n)$ is a compact subvariety, but will not typically be conjugate into $X$ or $Y$.  (In fact, $\supp(k_n)$ in general will not be displaceable, i.e. there will be no diffeomorphism $S$ such that $S(\supp(k_n)) \cap \supp(k_n) = \emptyset$, so one cannot hope to imitate the previous proof using Lemma \ref{lemma.avila.sets}.)  However, we can apply Theorem \ref{schrier-compact} to conclude that $\{k_n\}$ is generated by a finite set.  Thus, $\Diff^r_0(M)$ has the Schreier property.  
\end{proof}

\section{Further questions} \label{questions-sec}
We conclude with some natural questions for further study.  

Our argument in the proof of Proposition \ref{homeo-R prop} showed that every countable group in $\Homeo_0(\R)$ is contained in a group generated by 10 elements.   This bound is likely not optimal, but finding the optimal bound seems challenging.   More concretely, we ask
\begin{question} \label{gen-q}
Does there exist a countable set in $\Homeo_0(\R)$ that is not contained in a 2-generated subgroup?
\end{question}
Of course, by Proposition \ref{homeo-R prop}, it suffices to consider sets of cardinality 10.  
We note that the Higman embedding theorem shows that an abstract countable group can be embedded in one generated by two elements, and that Galvin \cite{Galvin} proved that this was also the case within the class of subgroups of the group of permutations of an infinite set.      Perhaps Question \ref{gen-q} is more approachable when $\Homeo_0(\R)$ is replaced by $\Diff_0(\R^n)$.  

It is also natural to ask for other transformation groups that satisfy (or fail to satisfy) strong distortion and Schreier's property.   We mentioned the groups $\Homeo(\mathbb{S}^2, \mathrm{area})$ and  $\Diff^r(\mathbb{S}^2, \mathrm{area})$ in the introduction as natural candidates.  We see no obvious obstruction to satisfying Schreier's property, but our proof tools do not apply here.    

Finally, we reiterate the open problem of strong boundedness for homeomorphism groups of manifolds with finite fundamental group.  The obvious first case is the following.  

\begin{question} 
Is $\Homeo(\R\mathrm{P}^2)$ strongly bounded?  If not, is there a natural, geometrically motivated length function on this group?  
\end{question}


\vspace{.5in}

Fr\'ed\'eric Le Roux 

Institut de Math\'ematiques de Jussieu - Paris Rive Gauche

Universit\'e Marie et Pierre Curie

4 place Jussieu, Case 247, 75252 Paris C\'edex 5

e-mail: frederic.le-roux@imj-prg.fr

\vspace{.3in}
Kathryn Mann

Dept. of Mathematics, Brown University

Box 1917, 
151 Thayer Street,

Providence, RI 02912

e-mail: mann@math.brown.edu


\end{document}